\documentclass[12pt]{amsart} % Specifies the document style.
\usepackage{amsmath} % symbols like equation*
\usepackage{amssymb} % symbols like mathbb
\usepackage{amsthm} % proof
\usepackage{mathtools} % norm
\usepackage[colorlinks]{hyperref} % url (also in references)
\usepackage{tikz}

\newtheorem{thm}{Theorem}

\newtheorem{lem}[thm]{Lemma}

\DeclarePairedDelimiter{\norm}{\lVert}{\rVert}

\newcommand{\sub}{\subseteq}

\newcommand{\R}{\mathbb{R}}

\newcommand{\N}{\mathbb{N}}
\newcommand{\eps}{\varepsilon}
\begin{document}

\pagestyle{myheadings} \thispagestyle{empty} \markright{}
\title{Sharp $\ell^q(L^p)$ decoupling for paraboloids}

\begin{abstract}
In this short expository note, we prove the following result, which is a special case of the main theorem in \cite{MR4541334}. For each $n\ge 2$ and $p,q\in [2,\infty]$, we prove upper bounds of $\ell^q(L^p)$ decoupling constants for paraboloids in $\R^n$, as well as presenting extremisers for each case. Both are sharp up to $\eps$-losses.
\end{abstract}

\author{Tongou Yang}
\address[Tongou Yang]{Department of Mathematics, University of California\\
Los Angeles, CA 90095, United States}
\email{tongouyang@math.ucla.edu}

\date{}
\maketitle

\section{Introduction}
Let $d\ge 2$ and $v_i\in \{-1,1\}$ for $i=1,\dots,d-1$. Write $\xi=(\xi_1,\dots,\xi_{d-1})$, $x'=(x_1,\dots,x_{d-1})$ and $x=(x',x_d)$. Define the paraboloid
\begin{equation}
    \phi(\xi)=\sum_{i=1}^{d-1}v_i \xi_i^2, 
\end{equation}
and its associated extension operator
\begin{equation}
    Eg(x_1,\dots,x_d)=\int_{[0,1]^{d-1}}g(\xi)e(x'\cdot \xi+x_d\phi(\xi))d\xi.
\end{equation}
For $\delta\in \N^{-1/2}$, let $\mathcal P(\delta)$ be a tiling of $[0,1]^{d-1}$ by cubes of side length $\delta^{1/2}$. 

For $p,q>0$ and each nonzero $g\in L^1([0,1]^{d-1})$, denote
\begin{equation}
    R(g)=R_{p,q}(g):=\sup_{B:\text{$d$-cube, }l(B)=\delta^{-1}}\frac {\norm{Eg}_{L^p(B)}}{\norm{\norm{E(g1_Q)}_{L^p(w_B)}}_{l^q(Q\in \mathcal P(\delta))}},
\end{equation}
where $w_B$ is a weight function adapted to $B$, such as (see \cite{BD2017_study_guide})
\begin{equation*}
    w_B(x)=\left(1+\frac {|x-c_B|}{\delta^{-1}}\right)^{-100d},
\end{equation*}
where $c_B$ is the centre of $B$. Let
\begin{equation}
    D_{p,q}(\delta)=\sup\{R(g):0\ne g\in L^1([0,1]^{d-1})\}.
\end{equation}
Denote
\begin{equation}
    p_d=\frac {2(d+1)}{d-1}, \quad N=\delta^{-\frac {d-1}2}.
\end{equation}
In the case of an elliptic paraboloid, Bourgain and Demeter \cite{BD2015} proved
\begin{thm}\label{thm_BD15}
Let $v_i=1$ for $i=1,\dots,d-1$. Then for every $p,q\in [2,\infty]$ we have
\begin{equation}
    D_{p,2}(\delta)\lessapprox \max\{1,N^{\frac 1 2-\frac {p_d}{2p}}\}.
\end{equation}
\end{thm}
Here and throughout this article, $A(\delta)\lessapprox B(\delta)$ will always mean that for every $\eps>0$, there exists some $C_\eps$ such that $A(\delta)\le C_\eps \delta^{-\eps}B(\delta)$ for every $\delta\in (0,1)$. We define $\gtrapprox$ and $\approx$ in a similar way.

For general $v_i=\pm 1$, Bourgain and Demeter \cite{BD2017} also proved
\begin{thm}\label{thm_BD17}
Let $p\in [2,\infty]$. Then
\begin{equation}
    D_{p,p}(\delta)\lessapprox \max\{N^{\frac 1 2-\frac 1 p},N^{1-\frac 1 p-\frac {p_d}{2p}}\}.
\end{equation}
\end{thm}

They also proved a version of $\ell^2(L^p)$ decoupling in the case of hyperbolic paraboloids:
\begin{thm}\label{thm_BD17_l2}
    Let $p\in [2,\infty]$. Let $d(v)$ denote the minimum number of positive and negative entries of $v_i$, $1\le i\le d-1$, and denote
    \begin{equation}
    p_{v}=\frac {2(d+1-d(v))}{d-1-d(v)}.
    \end{equation}
    Then 
    \begin{equation}
    D_{p,2}(\delta)\lessapprox 
    \begin{cases}
        N^{\frac {d(v)}{d-1}\left(\frac 1 2-\frac 1 p\right)},&\quad \text{if}\quad 2\le p\le p_v,\\
        N^{\frac 1 2-\frac {p_d}{2p}},&\quad \text{if}\quad p_v\le p\le \infty.
    \end{cases}
\end{equation}
\end{thm}
Note that $p_d\le p_v$, with equality if and only if $d(v)=0$, that is, we are in the case of elliptic paraboloids.

\subsection{Main results}
In the case of an elliptic paraboloid, we prove that
\begin{thm}\label{thm_main_elliptic}
Let $v_i=1$ for $i=1,\dots,d-1$. Then for all $p,q\in [2,\infty]$, we have
\begin{equation}
    D_{p,q}(\delta)\approx \max\{N^{\frac 1 2-\frac 1 q}, N^{1-\frac {p_d}{2p}-\frac 1 q} \}.
\end{equation}
\end{thm}
\begin{figure}
    \centering
    \begin{tikzpicture}
    \draw[->, thick] (0,0)--(5,0) node[right]{$\frac{1}{p}$};
    \draw[->, thick] (0,0)--(0,5) node[above]{$\frac{1}{q}$};
    \coordinate (A_1) at (0,0);
    \node[circle,fill,inner sep=1pt, label=left:{$A_1$}] at (A_1) {};
    \coordinate (A_2) at (0,8/2);
    \node[circle,fill,inner sep=1pt, label=left:{$A_2$}] at (A_2) {};
    \coordinate (A_3) at (8/6,8/2);
    \node[circle,fill,inner sep=1pt, label=above:{$A_3$}] at (A_3) {};
    \coordinate (A_4) at (8/4,8/4);
    \node[circle,fill,inner sep=1pt, label=right:{$A_4$}] at (A_4) {};
    \coordinate (A_5) at (8/4,0);
    \node[circle,fill,inner sep=1pt, label=below:{$A_5$}] at (A_5) {};
    \coordinate (A_6) at (8/2,0);
    \node[circle,fill,inner sep=1pt, label=below:{$A_6$}] at (A_6) {};
    \coordinate (A_7) at (8/2,8/2);
    \node[circle,fill,inner sep=1pt, label=right:{$A_7$}] at (A_7) {};

    \draw[] (A_4)--(A_5)--(A_6)--(A_7)--(A_4);
    \draw[] (A_3)--(A_4)--(A_7)--(A_3);
    \draw[] (A_1)--(A_2)--(A_3)--(A_4)--(A_5)--(A_1);

    \draw[] (-0.1,4) -- (0.1,4) node[pos = 0.5, anchor = south west] {$\frac{1}{2}$};
    \draw[] (4,-0.1) -- (4,0.1) node[pos = 0.5, anchor = south west] {$\frac{1}{2}$};
    \end{tikzpicture}
    \caption{The interpolation diagram}
    \label{fig:enter-label}
\end{figure}
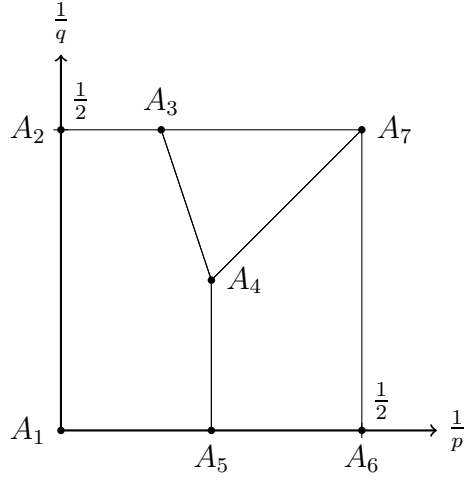

More generally, for hyperbolic paraboloids, we also have
\begin{thm}\label{thm_main_general}
In the $(1/p,1/q)$ interpolation diagram (see Figure \ref{fig:enter-label}), let 
\begin{align*}
    &A_1=\left(0,0\right),\quad A_2=\left(0,\frac 1 2\right), \quad A_3=\left(\frac 1 {p_v},\frac 1 2\right),\quad A_4=\left(\frac 1 {p_d},\frac 1 {p_d}\right)\\
    &A_5=\left(\frac 1 {p_d},0\right),\quad A_6=\left(\frac 1 2,0\right),\quad A_7=\left(\frac 1 2,\frac 1 2\right).
\end{align*}
Then the following holds:
\begin{equation}
    D_{p,q}(\delta)\approx 
    \begin{cases}
    & N^{\frac 1 2-\frac 1 q}, \quad \text{if $(1/p,1/q)$ lies in trapezoid $A_4 A_5 A_6 A_7$},\\
    & N^{\frac 1 2-\frac 1 q+\frac {d(v)}{d-1}\left(\frac 1 q-\frac 1 p\right)}, \quad \text{if $(1/p,1/q)$ lies in triangle $A_3 A_4 A_7$},\\
    & N^{1-\frac {p_d}{2p}-\frac 1 q}, \quad \text{if $(1/p,1/q)$ lies in pentagon $A_1 A_2 A_3 A_4 A_5$}.
    \end{cases}
\end{equation}
\end{thm}

{\it Acknowledgement.} The author thanks Jianhui Li for thoughtful discussions.

\section{extremisers for elliptic paraboloids}
\subsection{The constant test function}
\begin{thm}\label{thm_extreme_equiv_1}
Let $g\equiv 1$. For every $p\in [2,\infty]$, $q\in (0,\infty]$ we have
\begin{equation}
    R(g)\gtrsim N^{1-\frac {p_d}{2p}-\frac 1 q}.
\end{equation}
In particular, if $p\ge p_d$, then the constant function is an extremizer for $D_{p,q}(\delta)$ for elliptic paraboloids.
\end{thm}

\begin{proof}
If $p<\infty$, we use the equivalent neighbourhood version of decoupling. For each $Q$, let $f_Q$ be a Schwartz function such that $\widehat {f_Q}$ is essentially the normalised indicator function of the $\delta$-neighbourhood of the graph of $\phi$ over $Q$. Let $f=\sum_Q f_Q$. Let $B$ be centred at the origin and have side length $\delta^{-1}$.

For $|x|\le c\ll 1$, we have $|e(x'\cdot \xi+x_d\phi(\xi))|>1/2$. Thus $|f(x)|\sim N$ over $[-c,c]^{d-1}$, and hence
\begin{equation}
    \norm{f}_{L^p(B)}\geq \norm{f}_{L^p([-c,c]^{d-1})}\sim N.
\end{equation}
On the other hand, we have $f_Q\sim 1_{T_Q}$ where $T_Q$ is a rectangle with dimensions $(\delta^{-1/2})^{d-1}\times \delta^{-1}$. Thus
\begin{equation}
    \norm{f_Q}_{L^p(w_B)}\sim \delta^{-\frac {d-1}{2p}-\frac 1 p},
\end{equation}
which holds for every $Q$. Thus we have
\begin{align*}
      R(g)
      &\gtrsim \frac N {N^{\frac 1 q}\delta^{-\frac {d-1}{2p}-\frac 1 p}}=N^{1-\frac 1 q-\frac 1 p}\delta^{\frac 1 p}
      =N^{1-\frac {p_d}{2p}-\frac 1 q}.
\end{align*}
For $p=\infty$, it is easy to see that $\norm{E1}_{L^\infty(B)}= |E1(0)|\sim 1$, and $\norm{E1_Q}_{L^\infty(w_B)}= |E1_Q(0)|\sim |Q|=N^{-1}$. Thus we also have
\begin{equation}
    R(g)\gtrsim \frac {1}{N^{\frac 1 q-1}}=N^{1-\frac 1 q}.
\end{equation}
\end{proof}

\subsection{The exponential sum test function}

\begin{lem}\label{exponential_sum}
For $\delta\in \N^{-2}$ we define the function on $\mathbb T^2$ as
$$
f(x,y)=\sum_{j=1}^{\delta^{-1/2}}e\left(jx+ j^2 y\right).
$$
Then for all $2\leq p\leq \infty$, we have
$$
\norm{f}_{L^p(\mathbb T^2)}\approx
\begin{cases}
\delta^{-\frac 1 4} &\text{ if }2\leq p\leq 6\\
\delta^{-\frac 1 2+\frac {3}{2p} }&\text{ if }6\leq p\leq\infty
\end{cases}.
$$
\end{lem}
\begin{proof}
We first prove the upper bound by testing
$$
g(s)=\sum_{j=1}^{\delta^{-1/2}} \Delta_{j\delta^{1/2}}(s),
$$
where $\Delta_a$ is the delta-mass at $a$. (A more rigorous argument is to take an approximation to the identity at each of the delta masses.) Then we have
$$
Eg(x,y)=\sum_{j=1}^{\delta^{-1/2}}e\left(j\delta^{1/2}x+ j^2\delta y\right)=f(\delta^{1/2}x,\delta y).
$$
Also, for each $j$ we have
$$
Eg_j(x,y)=e\left(j\delta^{1/2}x+ j^2\delta y\right).
$$
Thus $\norm{Eg_j}_{L^p(w_B)}\sim \delta^{-\frac 2 p}$, and so 
$$
\norm{\norm {Eg_j}_{L^p(w_{B})}}_{l^2(j)}\sim \delta^{-\frac 1 4 }\delta^{-\frac 2 p}.
$$
But by periodicity, we have
$$
\norm{Eg}_{L^p(B)}=\delta^{-\frac 2 p}\norm f_{L^p(\mathbb T^2)}.
$$
The desired upper bound then follows from Theorem \ref{thm_BD15} with $d=2$.

The lower bound follows from \cite{MR2661311}. We provide a little more detail for clarity. The case $p=\infty$ is trivial, taking $x=y=0$. The case $6\leq p<\infty$ follows by considering $(x,y)$ near the origin; the detail can be found in Theorem 2.2 of \cite{Yip} (with $s=p/2$, and the proof there is easily seen to work for all real numbers $s>0$.) The case $p=2$ and $p=4$ follows from the first (trivial) bound of Theorem 2.3 of \cite{Yip}. 

Thus, the only case remaining is the case when $2< p< 6$ and $p\neq 4$. Assume $2<p<4$ first. Then $4\in (p,6)$ and using the log-convexity of $L^p$-norms, we have
$$
\norm f_{L^4(\mathbb T^2)}\leq \norm f^{1-\theta}_{L^p(\mathbb T^2)}\norm f^\theta_{L^6(\mathbb T^2)},
$$
where $\frac {1-\theta}{p}+\frac \theta 6=\frac 1 4$. But since $\norm f_{L^4(\mathbb T^2)}\sim \delta^{-1/4}$ and $\norm f_{L^6(\mathbb T^2)}\sim \delta^{-1/4}$, we also have $\norm f_{L^p(\mathbb T^2)}\gtrsim \delta^{-1/4}$. The case $4<p<6$ is similar.
\end{proof}

\begin{thm}\label{thm_exp_sum}
Let $g=\sum_{Q}\Delta(c_Q,\phi(c_Q))$ be the sum of delta-masses at $(c_Q,\phi(c_Q))$. Then for every $p\in [2,\infty]$, $q\in (0,\infty]$ we have
\begin{equation}
    R(g)\gtrsim \max\{N^{\frac 1 2-\frac 1 q}, N^{1-\frac {3}{p}-\frac 1 q} \}.
\end{equation}
As a result, if $2\le p\le p_d$, then $g$ is an extremizer for $D_{p,q}(\delta)$ for elliptic paraboloids.
\end{thm}
{\it Remark.} For $d\ge 3$ and $p> p_d$, the bound $N^{1-\frac {3}{p}-\frac 1 q}$ does not agree with the sharp bound $N^{1-\frac {p_d}{2p}-\frac 1 q}$. Thus to get sharp decoupling, we will need the constant test function. Thus, only the first lower bound $N^{\frac 1 2-\frac 1 q}$ will be useful in our proof later. Nevertheless, if $d=2$, $p_d=6$, then $g$ is an extremizer for every $p\ge 2$. 

\begin{proof}
Let $B$ be centred at the origin and have side length $\delta^{-1}$. We have $|E(g1_Q)(x)|\equiv 1$, and thus
\begin{equation}
    \norm{\norm{E(g1_Q)}_{L^p(w_B)}}_{l^q(Q\in \mathcal P(\delta))}\sim N^{\frac 1 q}|B|^{\frac 1 p}=N^{\frac 1 q+\frac {2d}{p(d-1)}}.
\end{equation}
For the left hand side, we have
\begin{align*}
    Eg(x)
    &=\sum_Q e(x' c_Q+x_d \phi(c_Q))\\
    &=\prod_{i=1}^{d-1}\sum_{j_i=1}^{\delta^{-1/2}}e(j_i \delta^{1/2}x_i+v_i j_i^2\delta x_d).
\end{align*}
For $p<\infty$,
\begin{align}
    &\int_B |Eg(x)|^p dx\nonumber\\
    &=\int_{-\delta^{-1}}^{\delta^{-1}}\left(\prod_{i=1}^{d-1}\int_{-\delta^{-1}}^{\delta^{-1}}\left|\sum_{j_i=1}^{\delta^{-1/2}}e(j_i \delta^{1/2}x_i+v_i j_i^2\delta x_d)\right|^p dx_i\right) dx_d\nonumber\\
    &=\delta^{1+\frac {d-1}{2}}\int_{-1}^{1}\left(\prod_{i=1}^{d-1}\int_{-\delta^{-1/2}}^{\delta^{-1/2}}\left|\sum_{j_i=1}^{\delta^{-1/2}}e(j_i x_i+v_i j_i^2 x_d)\right|^p dx_i\right) dx_d\nonumber\\
    &=\delta^{-d}\int_{-1}^{1}\left(\prod_{i=1}^{d-1}\int_{-1}^{1}\left|\sum_{j_i=1}^{\delta^{-1/2}}e(j_i x_i+v_i j_i^2 x_d)\right|^p dx_i\right) dx_d,
\end{align}
where the last line follows from periodicity. By taking complex conjugation and symmetry, it is easy to see that for each $i$ we have
\begin{align*}
    \int_{-1}^{1}\left|\sum_{j_i=1}^{\delta^{-1/2}}e(j_i x_i+v_i j_i^2 x_d)\right|^p dx_i
    =\int_{-1}^{1}\left|\sum_{j=1}^{\delta^{-1/2}}e(j t+ j^2 x_d)\right|^p dt:=S_p(x_d).
\end{align*}
Thus
\begin{equation}
    \int_B |Eg(x)|^p dx=\delta^{-d}\int_{-1}^1 S_p(x_d)^{d-1}dx_d.
\end{equation}
In particular, this computation shows that the test function $g$ disregards the signs of $v_i$, and thus it should only extremize the case of elliptic paraboloids.

We then use Jensen's inequality:
\begin{align*}
     &\int_{-1}^1 S_p(x_d)^{d-1}dx_d\\
     &\gtrsim \left(\int_{-1}^1 S_p(x_d)dx_d\right)^{d-1}\\
     &\sim \max\{(\delta^{-\frac 1 2})^{\frac p 2(d-1)}, (\delta^{-\frac 1 2})^{ (p-3) (d-1)} \}\\
     &=\max\{N^{\frac p 2}, N^{p-3} \},
\end{align*}
where the $\sim$ follows from Lemma \ref{exponential_sum}. This gives
\begin{equation}
    \norm{Eg}_{L^p(B)}\gtrsim \delta^{-\frac d p}\max\{N^{\frac 1 2}, N^{1-\frac 3 p}\}.
\end{equation}
If $p=\infty$, it is also easy to see that $\norm{Eg}_{L^\infty(B)}\ge |Eg(0)|\sim N$, which agrees with the expression above.

As a result, we have
\begin{equation}
    R(g)\gtrsim \frac {\delta^{-\frac d p}\max\{N^{\frac 1 2}, N^{1-\frac 3 p}\}}{N^{\frac 1 q+\frac {2d}{p(d-1)}}}=\max\{N^{\frac 1 2-\frac 1 q}, N^{1-\frac {3}{p}-\frac 1 q} \}.
\end{equation}
\end{proof}
%{\it Remark 1.} In the proof, we have used Jensen's inequality, but obtained a sharp result. Compared with the upper bound $D_{p,q}(\delta)\approx \max\{N^{\frac 1 2-\frac 1 q}, N^{1-\frac {p_d}{2p}-\frac 1 q} \}$ by Bourgain and Demeter, this shows that $S_p(x_d)$ is essentially constant for $|x_d|\le 1$, when $2\le p\le p_d$.

{\it Remark.} When $d=2$, $p=6$, a more refined number-theoretic argument (see \cite{MR2661311} and Theorem 3.5 of \cite{Yip}) proves $\norm{Eg}_{L^p(B)}\gtrsim N^{7/6}(\log N)^{1/6}$, implying $D_{6,2}(\delta)\gtrsim |\log \delta|^{1/6}$. So in general, we must have at least logarithmic losses in the upper bound. In fact, in \cite{GMW24}, the authors proved that when $d=2$, $p=6$, we can have $D_{6,2}(\delta)\lesssim |\log \delta|^{O(1)}$. 

\subsection{Proof of Theorem \ref{thm_main_elliptic}}
\begin{proof}
The upper bound follows from Theorem \ref{thm_BD15} and H\"older's inequality. The lower bound follows from Theorems \ref{thm_extreme_equiv_1} and \ref{thm_exp_sum}.
\end{proof}

\section{extremisers for hyperpolic paraboloids}
We are now going to prove Theorem \ref{thm_main_general}. By renaming the variables, we may let 
\begin{equation}
    \phi(\xi)=\xi_1^2+\cdots+\xi^2_{d(v)}-\xi^2_{d(v)+1}-\cdots-\xi_{d-1}^2,
\end{equation}
where we recall $d(v)\le (d-1)/2$ is the minimum number of positive and negative entries of $v_i$, $1\le i\le d-1$.

First, if $p\ge p_v$, then since $p_v\ge p_d$, by Theorem \ref{thm_BD17_l2} with H\"older's inequality and Theorem \ref{thm_extreme_equiv_1}, our claim follows. Thus it suffices to consider $2\le p\le p_v$.

\subsection{The hyperplane test function}
\begin{thm}\label{thm_hyperplane}
    Let $p,q\in [2,\infty]$ and $h$ be an extremizer for the decoupling for the elliptic paraboloid in dimension $d-1-2d(v)$ (set $h\equiv 1$ if $d-1-2d(v)=0$). Define    
    $g(\xi)=h(\xi_{2d(v)+1},\cdots,\xi_{d-1})\prod_{j=1}^{d(v)}\Delta(\xi_j-\xi_{d(v)+j})$. Then
    \begin{equation}
        R(g)\gtrsim N^{\frac 1 2-\frac 1 q+\frac {d(v)}{d-1}(\frac {1}{p}-\frac {1}{q})}.
    \end{equation}
\end{thm}
\begin{proof}
Let $B$ be centred at the origin and have side length $\delta^{-1}$. By the separability of $Eg$, we can easily compute that 
\begin{equation}
    Eg(x)=\tilde Eh(x_{2d(v)+1},\dots,x_{d-1},x_d)\prod_{k=1}^{d(v)}\left(\sum_{j_k=1}^{\delta^{-1/2}} e(\delta^{\frac 1 2}j_k(x_{k}+x_{k+d(v)}))\right),
\end{equation}
where $\tilde E$ is the extension operator for the elliptic paraboloid over the unit cube in $d-1-2d(v)$ dimensions.
We first compute
\begin{align*}
    &\int_{[-\delta^{-1},\delta^{-1}]^{2d(v)}}\prod_{k=1}^{d(v)}\left|\sum_{j_k=1}^{\delta^{-1/2}} e(\delta^{\frac 1 2}j_k(x_{k}+x_{k+d(v)}))\right|^p dx_1\cdots dx_{2d(v)}\\
    &=\left(\int _{[-\delta^{-1},\delta^{-1}]^{2}}\left|\sum_{j=1}^{\delta^{-1/2}} e(\delta^{\frac 1 2}j(x+y))\right|^p dxdy\right)^{d(v)}\\
    &\sim \left(\int_{-\delta^{-1}}^{\delta^{-1}}\delta^{-1}\left|\sum_{j=1}^{\delta^{-1/2}} e(\delta^{\frac 1 2}jz)\right|^p dz\right)^{d(v)}\\
    &\sim \delta^{-\frac {(p+3)d(v)}2}.
\end{align*}
Thus
\begin{equation}
    \norm{Eg}_{L^p(B)}\sim \delta^{-\frac {(p+3)d(v)}{2p}}\norm{\tilde Eh}_{L^p(\tilde B)},
\end{equation}
where $\tilde B\sub \R^{d-2d(v)}$ is the cube centred at $0$ and has side length $\delta^{-1}$.

Similarly, given each $Q:=\prod_{l=1}^{d-1}I_l\sub [0,1]^{d-1}$ of side length $\delta^{1/2}$. Let $\tilde Q=\prod_{l=2d(v)+1}^{d-1}I_l$, we see $E(g1_Q)(x)$ is nonzero only if and for each $1\le l\le d(v)$ we have $I_l=I_{l+d(v)}$. Thus
\begin{equation}
    |E(g1_Q)(x)|=|\tilde E_{\tilde Q
    }h(x_{2d(v)+1},\dots,x_{d-1},x_d)|\prod_{l=1}^{d(v)}1_{I_l=I_{l+d(v)}}.
\end{equation}
For such $Q$ we have
\begin{equation}
    \norm {E(g1_Q)}_{L^p(B)}\sim \norm {\tilde E_{\tilde Q}h}_{L^p(\tilde B)}\delta^{-\frac{2d(v)}{p}}.
\end{equation}
Thus
\begin{equation}
    \norm{\norm {E(g1_Q)}_{L^p(B)}}_{l^q(Q)}\sim \norm{\norm {\tilde E_{\tilde Q}h}_{L^p(\tilde B)}}_{l^q(\tilde Q)}\delta^{-\frac{2d(v)}{p}}\delta^{-\frac {d(v)}{2q}}.
\end{equation}
For $\tilde Eh$, the critical exponent this time becomes $p_{d-2d(v)}$ instead of $p_d$. Since $p\le p_v\le p_{d-2d(v)}$, we use Theorem \ref{thm_exp_sum}. Since $h$ is defined to be an extremizer, we have
\begin{equation}
        \norm {\tilde Eh}_{L^p(\tilde B)}\sim (\delta^{-\frac {d-1-2d(v)} 2})^{\frac 1 2-\frac 1 q}\norm{\norm{\tilde E_{\tilde Q
    }h}_{L^p(\tilde B)}}_{l^q(\tilde Q)}
    \end{equation}
    This implies that
    \begin{align*}
        R(g)
        &\gtrsim \delta^{-\frac {(p+3)d(v)}{2p}}\delta^{\frac {2d(v)}{p}}\delta^{\frac{d(v)}{2q}}(\delta^{-\frac {d-1-2d(v)} 2})^{\frac 1 2-\frac 1 q}\\
        &=N^{\frac 1 2-\frac 1 q}N^{\frac {d(v)}{d-1}(\frac {1}{q}-\frac {1}{p})}.
    \end{align*}
    
\end{proof}

\subsection{The case $p_d\le p\le p_v$}
Combining the sharp lower bounds in Theorems \ref{thm_extreme_equiv_1} and \ref{thm_hyperplane}, we get
\begin{equation}
    D_{p,q}(\delta)\gtrsim \max\{N^{1-\frac {p_d}{2p}-\frac 1 q},N^{\frac 1 2-\frac 1 q}N^{\frac {d(v)}{d-1}(\frac {1}{q}-\frac {1}{p})}\}.
\end{equation}
We need to determine which one is larger. Setting
\begin{equation}
    1-\frac {p_d}{2p}-\frac 1 q=\frac 1 2-\frac 1 q+\frac {d(v)}{d-1}\left(\frac {1}{q}-\frac {1}{p}\right),
\end{equation}
that is,
\begin{equation}
    \frac 1 2-\frac {p_d}{2p}=\frac {d(v)}{d-1}\left(\frac {1}{q}-\frac {1}{p}\right),
\end{equation}
we see this corresponds to the critical line $l_1$ that passes through the points $(1/p_v,1/2)$ and $(1/p_d,1/p_d)$ in the $(1/p,1/q)$ interpolation diagram.

\begin{itemize}
    \item If $(1/p,1/q)$ lies above $l_1$, then we have
    \begin{equation}
        D_{p,q}(\delta)\gtrsim N^{\frac 1 2-\frac 1 q+\frac {d(v)}{d-1}(\frac {1}{p}-\frac {1}{q})}.
    \end{equation}
We claim that $\lessapprox$ also holds in this case. By interpolation, it suffices to prove upper bounds at the following three points:
\begin{equation}
    \left(\frac 1 {p_v},\frac 1 2\right),\quad \left(\frac 1 {p_d},\frac 1 {p_d}\right),\quad \left(\frac 1 {p_d},\frac 1 2\right).
\end{equation}
The point $(\frac 1 {p_v},\frac 1 2)$ has been settled by the case $p\ge p_v$, since we can check by direct computation that
\begin{equation}
    \frac {d(v)}{d-1}\left(\frac 1 2-\frac 1 {p_v}\right)=\frac 1 2-\frac {p_d}{p_v}.
\end{equation}
The upper bound at $(\frac 1 {p_d},\frac 1 {p_d})$ follows from Theorem \ref{thm_BD17} when $p=p_d$. The upper bound at $(\frac 1 {p_d},\frac 1 2)$ follows from Theorem \ref{thm_BD17_l2} when $p=p_d$.

\item If $(1/p,1/q)$ lies below $l_1$, then we have
    \begin{equation}
        D_{p,q}(\delta)\gtrsim N^{1-\frac {p_d}{2p}-\frac 1 q}.
    \end{equation}
    We have $\lessapprox$ also holds in this case, by the upper bounds at $(\frac 1 {p_v},\frac 1 2)$ and $(\frac 1 {p_d},\frac 1 {p_d})$, together with H\"older's inequality. 
\end{itemize}

\subsection{The case $2\le p\le p_d$}
Combining the sharp lower bounds in Theorems \ref{thm_exp_sum} and \ref{thm_hyperplane}, we get
\begin{equation}
    D_{p,q}(\delta)\gtrsim \max\{N^{\frac 1 2-\frac 1 q},N^{\frac 1 2-\frac 1 q}N^{\frac {d(v)}{d-1}(\frac {1}{q}-\frac {1}{p})}\}.
\end{equation}
It is easy to see that the critical line $l_2$ in this case is given by $1/p=1/q$.

\begin{itemize}
    \item If $1/p\le 1/q$, then we have
    \begin{equation}
        D_{p,q}(\delta)\gtrsim N^{\frac 1 2-\frac 1 q+\frac {d(v)}{d-1}(\frac {1}{p}-\frac {1}{q})}.
    \end{equation}
We have $\lessapprox$ also holds in this case, which follows immediately by interpolating the upper bounds we already obtained at $(\frac 1 {p_d},\frac 1 {p_d})$ and $(\frac 1 {p_d},\frac 1 {2})$ with $(\frac 1 {2},\frac 1 {2})$ given by Plancherel.

\item If $1/p\ge 1/q$, then we have
    \begin{equation}
        D_{p,q}(\delta)\gtrsim N^{\frac 1 2-\frac 1 q}.
    \end{equation}
We have $\lessapprox$ also holds in this case, by the upper bound at $(\frac 1 {p_d},\frac 1 {p_d})$ and Plancherel at $(\frac 1 {2},\frac 1 {2})$, together with H\"older's inequality. 
\end{itemize}

\bibliographystyle{alpha}
\bibliography{reference}

\end{document}